\documentclass[12pt]{amsart}
\usepackage{amsfonts, amsmath, amsthm, amssymb}
\usepackage{url, fancyhdr}
\usepackage[all]{xy}

\input xy
\xyoption{all}

\newtheorem{thm}{Theorem}[section]
\newtheorem{lem}[thm]{Lemma}

\newtheorem{prop}[thm]{Proposition}
\newtheorem{claim}[thm]{Claim}

\theoremstyle{definition}
\newtheorem{example}[thm]{Example}
\newtheorem{rmk}[thm]{Remark}
\newtheorem{defn}[thm]{Definition}

\newenvironment{axiom}{\begin{list}{$\bullet$}{\setlength{\labelsep}{.7cm}
\setlength{\leftmargin}{2.5cm}\setlength{\rightmargin}{0cm}%
\setlength{\labelwidth}{1.8cm}\setlength{\itemsep}{0pt}}}{\end{list}}
\newcommand{\ax}[1]{\item[{\bf #1}\hfill]\index{#1}}

\def\bN{\mathbb{N}}
\def\bQ{\mathbb{Q}}
\def\bR{\mathbb{R}}

\newcommand{\cupdot}{\mathbin{\mathaccent\cdot\cup}}

\begin{document}

\title[Partial metric spaces with negative distances and fixed point theorems]{Partial metric spaces with negative distances \\and fixed point theorems}
\author{Samer Assaf}
\address{Department of Mathematics and Statistics,
University of Saskatchewan, 106 Wiggins Road,
Saskatoon, Saskatchewan, Canada S7N 5E6}
\email{ska680@mail.usask.ca}
\author{Koushik Pal}
\address{Department of Mathematics and Statistics,
University of Saskatchewan, 106 Wiggins Road,
Saskatoon, Saskatchewan, Canada S7N 5E6}
\email{koushik.pal@usask.ca}
\thanks{The second-named author is partially supported by a PIMS Postdoctoral Fellowship.}


\begin{abstract}\noindent
In this paper we consider partial metric spaces in the sense of O'Neill. We introduce the notions of strong partial metric spaces and Cauchy functions. We prove a fixed point theorem for such spaces and functions that improves Matthews' contraction mapping theorem in two ways. First, the existence of fixed points now holds for a wider class of functions and spaces. Second, our theorem also allows for fixed points with nonzero self-distances. We also prove fixed point theorems for orbitally $r$-contractive and orbitally $\phi_r$-contractive maps. We then apply our results to give alternative proofs of some of the other known fixed point theorems in the context of partial metric spaces.
\end{abstract}
\maketitle

\section{Introduction}
The notion of {\em distance} is fundamental in mathematics and variations on distance have been much studied (see \cite{MMDED}). One such variation, the {\em partial metric}, was introduced by Matthews (see \cite{SM1, SM2}). It differs from a metric in that points are allowed to have nonzero ``self-distances'' (i.e., $d(x, x)\ge 0$) and the triangle inequality is modified to account for positive self-distance. The notion of a partial metric has been both fruitful and well-studied (see \url{http://www.dcs.warwick.ac.uk/pmetric/index.html}). 

O'Neill \cite{SJON} extended Matthews' definition to allow partial metrics with ``negative distances''. In this paper, we study partial metrics in the sense of O'Neill. If $X$ is a nonempty set and $\bR$ is the set of real numbers, then a {\bf partial metric}, in the sense of O'Neill, is a function $p:X\times X\to\bR$ satisfying for all $x, y, z\in X$:
%
\begin{axiom}
\ax{(sep)} $p(x, x) = p(x, y) = p(y, y) \iff x = y$,
\ax{(ssd)} $p(x, x)\le p(x, y)$,
\ax{(sym)} $p(x, y) = p(y, x)$,
\ax{(ptri)} $p(x, y)\le p(x, z) + p(z, y) - p(z, z)$.
\end{axiom}
These four conditions are called the {\em separation condition}, the {\em small self-distance condition}, the {\em symmetry condition} and the {\em partial metric triangle inequality condition}, respectively. 

The implications of Matthew's {\bf(ssd)} axiom was later studied by Heckmann \cite{RH} in his work on {\em weak partial metrics}, where he, in fact, drops that axiom. In this paper, we take a different route and strengthen the {\bf(ssd)} axiom and study its consequences. To that effect, we introduce the following definition.

\begin{defn}
Let $(X, p)$ be a partial metric space. We say $(X, p)$ is {\em strong} if conditions {\bf(sep)} and {\bf(ssd)} are replaced by the {\em strictly small self-distance condition}
\begin{axiom}
\ax{(sssd)} $p(x, x) < p(x, y)$ \;\;for all $x, y\in X$ with $x \not= y$.
\end{axiom}
And we say $(X, p)$ is {\em bounded below by $r_0$}, for some $r_0\in\bR$, if $(X, p)$ satisfies the {\em lower bound condition}
\begin{axiom}
\ax{(lbd)} $p(x, y) \ge r_0$ \;\;for all $x, y\in X$.
\end{axiom}
\end{defn}

Three remarks follow this definition. First, the authors intend to approach this paper more from a mathematical perspective than a computer science one. So the usual poset structure $(X, \sqsubseteq_p)$ on a partial metric space $(X, p)$ defined by $x \sqsubseteq_p y :\iff p(x, x) = p(x, y)$ is not relevant to the current paper. But for the sake of completeness, we should mention that a consequence of our axiom {\bf(sssd)} is that the relation $\sqsubseteq_p$ is, in fact, equality.

Second, it is easy to see that if $(X, p)$ is a partial metric space bounded below by $r_0$, then $(X, p_{r_0})$ is a partial metric space bounded below by zero, where $p_{r_0}(x, y) := p(x, y) - r_0$. That is to say, a ``bounded below partial metric'' can be linearly adjusted to obtain an equivalent partial metric. So this definition might not seem terribly interesting at the first sight. But the real power of this definition shows in the Theorems \ref{cauchyorbcontnexpfpt}, \ref{cauchyorbcontractivefpt} and \ref{phircontmapthm}, which are not mere linear adjustments of the corresponding theorems known in the context of partial metric spaces.

Third, it should be noted that even though we don't mention the ``{\em separation}'' condition {\bf(sep)} explicitly for strong partial metric spaces, it follows from condition {\bf(sssd)} : if $x \not= y$, then $p(x, x) < p(x, y)$; therefore, $p(x, x) = p(x, y)\implies x = y$. So, strong partial metric spaces are indeed partial metric spaces. Also, it is easy to see that a metric space is just a strong partial metric space with all self-distances zero. In short,
$$\{\mbox{Metric Spaces}\}\subseteq\{\mbox{Strong Partial Metric Spaces}\}\subseteq\{\mbox{Partial Metric Spaces}\}.$$
In Section 3, we provide examples to show that these inclusions are strict, cf. Remark~\ref{counterexample3}.

One of our main goals in this paper is to prove fixed point theorems for partial metric spaces. To do that, one usually needs to place some conditions on the function and/or on the underlying space. To that end, we use the following definitions given by Matthews  \cite{SM1, SM2}, although in his case (unlike ours) the number $r$ in the definition is necessarily nonnegative.

\begin{defn}		\label{cauchylimit}
Let $(X, p)$ be a partial metric space. A sequence $\langle x_n\rangle_{n\in\bN}$ is called a {\em Cauchy sequence in $(X, p)$} if there is some $r\in\bR$ such that
$$\lim_{m, n\to\infty}p(x_m, x_n) = r.$$
An element $a\in X$ is called a {\em limit} of the sequence $\langle x_n\rangle_{n\in\bN}$ if
$$\lim_{n\to\infty}p(a, x_n) = p(a, a).$$
An element $a\in X$ is called a {\em special limit} of the sequence $\langle x_n\rangle_{n\in\bN}$ if
$$\lim_{m, n\to\infty}p(x_m, x_n) = \lim_{n\to\infty}p(a, x_n) = p(a, a).$$
Finally, a partial metric space $(X, p)$ is called {\em complete} if every Cauchy sequence $\langle x_n\rangle_{n\in\bN}$ in $(X, p)$ converges to a special limit $a\in X$.
\end{defn}

Although a limit of a sequence is not unique in general in a partial metric space, we will show later that a special limit of a Cauchy sequence, if it exists, is in fact unique, cf.~Lemma~\ref{spllimitunique}. We also need some conditions on the function as given by the following definition. A remark about the notation: we write $fx$ for $f(x)$.

\begin{defn}		\label{contmap}
Let $(X, p)$ be a partial metric space and $f:X\to X$ be a map. We say $f$ is {\em non-expansive} if it satisfies
$$p(fx, fy)\le p(x, y)\;\;\;\;\;\;\forall x, y \in X.$$
If $(X, p)$ is bounded below by zero, we say $f$ is {\em contractive} if for some $0\le c < 1$ we have
$$p(fx, fy) \le c\,p(x, y)\;\;\;\;\;\;\forall x, y \in X.$$
\end{defn}

Matthews proved the following contraction mapping theorem for partial metric spaces.

\begin{thm}\cite[Theorem 5.3]{SM2}			\label{contmapthm}
For each complete partial metric space $(X, p)$ bounded below by zero, and for each contractive map $f:X\to X$, there exists a unique $a\in X$ such that $fa = a$ and, moreover, $p(a, a) = 0$.
\end{thm}

In this paper, we improve on this theorem in two ways. First, we weaken the condition on the function (cf. Theorem~\ref{cauchyorbcontnexpfpt}) and the underlying space (allowing negative distances) so that the existence of a fixed point now holds for a wider class of functions and spaces, albeit we lose uniqueness. Second, we allow a condition on the function general enough to admit fixed points with nonzero self-distances.

\begin{defn}		\label{cauchyatx0}
Given a partial metric space $(X, p)$, an element $x_0\in X$, and a map $f:X\to X$, we say $f$ is {\em Cauchy at $x_0$} if the orbit $\langle f^nx_0\rangle_{n\in\bN}$ of $x_0$ under $f$ is a Cauchy sequence in $(X, p)$. We say $f$ is {\em Cauchy at $x_0$ with special limit $a$} if $f$ is Cauchy at $x_0$ and $a$ is a special limit of the Cauchy sequence $\langle f^nx_0\rangle_{n\in\bN}$. Finally, we say $f$ is {\em Cauchy} if $f$ is Cauchy at every $x\in X$.
\end{defn}
We will show later that if $(X, p)$ is bounded below by zero and $f$ is contractive, then $f$ is Cauchy, but not conversely. Thus, Cauchy maps form a strictly wider class than contractive maps even for partial metric spaces bounded below by zero. We also need the following definition.

\begin{defn}		\label{orbcont}
Given a partial metric space $(X, p)$, elements $x_0, z_0\in X$, and a map $f:X\to X$, we say $f$ is {\em orbitally continuous at $x_0$ for $z_0$} if
$$z_0\mbox{ is a limit of } \langle f^nx_0\rangle_{n\in\bN} \implies fz_0 \mbox{ is a limit of } \langle f^nx_0\rangle_{n\in\bN}$$
i.e.,
$$\lim_{n\to\infty} p(f^nx_0, z_0) = p(z_0, z_0) \implies \lim_{n\to\infty} p(f^nx_0, fz_0) = p(fz_0, fz_0).$$

We say $f$ is {\em orbitally continuous at $x_0$} if $f$ is orbitally continuous at $x_0$ for every $z\in X$. And, we say $f$ is {\em orbitally continuous} if it is orbitally continuous at every $x\in X$.
\end{defn}

The following are the main results of this paper.
\begin{thm}[Cauchy Mapping Theorem for Partial Metric Spaces]		\label{cauchyorbcontnexpfpt}
Let $(X, p)$ be a partial metric space, $x_0\in X$ be an element, and $f:X\to X$ be a map such that $f$ is Cauchy at $x_0$ with special limit $a\in X$. Further assume at least one of the following conditions holds:
\begin{enumerate}
\item $f$ is non-expansive and orbitally continuous at $x_0$ for $a$;
\item $f$ is orbitally continuous at $x_0$ for $a$ and $(X, p)$ is bounded below by $p(fa, fa)$;
\item $f$ is non-expansive and $(X, p)$ is bounded below by $p(a, a)$.
\end{enumerate}
Then $a$ is a fixed point of $f$.
\end{thm}

As it turns out, things are much simpler for a strong partial metric space as we need fewer conditions for a fixed point to exist.
\begin{thm}[Cauchy Mapping Theorem for Strong Partial Metric Spaces]		\label{strongcauchyorbcontnexpfpt}
Let $(X, p)$ be a strong partial metric space, $x_0\in X$ be an element, and $f:X\to X$ be a map such that $f$ is Cauchy at $x_0$ with special limit $a\in X$. Further assume at least one of the following conditions holds:
\begin{enumerate}
\item $f$ is non-expansive;
\item $f$ is orbitally continuous at $x_0$ for $a$.
\end{enumerate}
Then $a$ is a fixed point of $f$.
\end{thm}

The paper is organized as follows. In Section 2, we give a motivational example for studying nonzero self-distances as well as condition ({\bf sssd}). In Section 3, we describe the natural topology on a partial metric space, and show that it is $T_0$ in general and $T_1$ for a strong partial metric space. Section 4 gives an equivalent condition for orbitally continuous maps. In Section 5, we prove our two main results, which serve as the central tools in proving other fixed point theorems in the context of (strong) partial metric spaces. In Sections 6 and 7, we give examples of two classes of functions, namely ``orbitally $r$-contractive'' and ``orbitally $\phi_r$-contractive'' functions, that satisfy the property of being Cauchy and thereby provide more examples of fixed point theorems. Finally, in Section 8, we apply our main results to give more streamlined and concise proofs of some other known fixed point theorems from \cite{EKIE}, though in the context of strong partial metric spaces.
\newline\newline
{\bf Acknowledgements.}
The authors would like to thank Franz-Viktor Kuhlmann and Ed Tymchatyn for their careful proof-reading of this manuscript and for their numerous suggestions on corrections and improvements. The first-named author would also like to thank Katarzyna Kuhlmann for introducing him to generalized notions of metric spaces.

\section{An example in biological setting}
DNA, proteins, words are all examples of finite sequences generated from a finite alphabet. And often a generic question is: given two finite sequences $x=\langle x_1, x_2, \ldots, x_n\rangle$ and $y=\langle y_1, y_2, \ldots, y_m\rangle$, how similar are these two sequences?

In the case of DNA, for example, and while studying mutation from $x$ to $y$, it becomes important to come up with a measure that can effectively compare partial DNA strands. One such  measure is the following commonly used {\em scoring scheme} \cite{IEIJWT}: one first aligns two given words so that their lengths match (to do this one uses ``---'' as part of the alphabet), and then compares them letter-by-letter and assigns a score to each of the four distinct possibilities, namely, a score of $\alpha$ if the two letters match, a score of $\beta$ if they don't match, a score of $\gamma$ if only one of the letters is ``---'', and a score of zero if both the letters are ``---''. Then these scores are summed up to assign a total score for that particular alignment of the given pair of words. Finally all possible alignments are considered between the two words and the highest possible score is assigned to the pair, which is then used as a measure of the similarity or dissimilarity of the two words. As an example, if $\alpha = +1$, $\beta = -1$ and $\gamma = -2$, then the total score of the pair (CGATC, CAGA) for the particular alignment (x=CGA---TC, y=C---AGA---) is $+1 - 2 + 1 -2 -1 -2 = -5$. It is not hard to show that the best possible score for the same pair of words is $-2$ arising from the alignment (x=CGATC, y=C---AGA).

Fix $\alpha, \beta, \gamma\in\bR$ and let $s(x, y)$ denote the scoring function as mentioned above. We now show that for certain choices of $\alpha$, $\beta$ and $\gamma$, the scoring function $s(x, y)$ gives rise to a strong partial metric $p(x, y)$. 
\begin{prop}
The function $p(x, y) := -s(x, y)$ (where $s$ is as defined above) is a strong partial metric provided $\alpha > \beta$, $\alpha > \gamma$, $\beta\ge 2\gamma$ and $\gamma < 0$.
\end{prop}
\begin{proof}
It is easy to see that $s(x, y)$ (and hence $p(x, y)$) is symmetric for all nonempty words $x$ and $y$. Also, notice that for any given word $x$, the value $s(x, x)$ is the highest possible value of $s(x, y)$ for any word $y$ as $\alpha > \beta$, $\alpha > \gamma$ and $\gamma < 0$. In particular, for any given words $x$ and $y$ with $x\not= y$, we have $s(x, x) > s(x, y)$ (and hence $p(x, x) < p(x, y)$), since any insertion or deletion or mismatch with $x$ will irreparably reduce the value of $s(x, y)$ from the highest possible value of $s(x, x)$. 

So it suffices to demonstrate condition {\bf(ptri)}. To that end, let $x=\langle x_1, \ldots, x_n\rangle$, $y=\langle y_1, \ldots, y_n\rangle$ and $z=\langle z_1, \ldots, z_n\rangle$ be three word sequences of the same length such that $x$ and $z$ are optimally aligned. We can also assume without loss of generality that $y$ and $z$ are optimally aligned by aligning ``---'' with ``---'' if necessary. Since $x$ and $y$ need not be optimally aligned, we denote the score for this alignment of $x$ and $y$ by $h$. In particular, $h(x, y)\le s(x, y)$. We now show that 
\begin{eqnarray}	\label{eqn1}
h(x, y)\ge s(x, z) + s(z, y) - s(z, z).
\end{eqnarray} 
We have to consider several cases. Fix $1\le i\le n$. Then one of the following holds:
\begin{enumerate}
\item[i.] $x_i = y_i$, $z_i\not= $ ---
\item[ii.] $x_i = y_i$, $z_i = $ ---
\item[iii.] $x_i\not= y_i$, $z_i\not= $ ---
\item[iv.] $x_i\not= y_i$, $z_i = $ ---
\item[v.] $x_i = $ ---, $y_i = z_i$
\item[vi.] $x_i = $ ---, $y_i\not= z_i$
\item[vii.] $x_i = z_i = $ ---, $y_i\not= $ ---
\item[viii.] $x_i = y_i = $ ---, $z_i\not= $ ---
\item[ix.] $x_i = y_i = z_i = $ ---
\end{enumerate}
In the above list, when we say $a = b$ or $a\not= b$, we mean both $a$ and $b$ are letters other than ``---''. We leave it to the reader to verify that (\ref{eqn1}) holds in all these 9 cases. As examples, we work out three of these cases here, namely (i), (iv) and (viii):
\begin{itemize}
\item[i.] In this situation, we have $h(x_i, y_i) = \alpha$ and $s(z_i, z_i) = \alpha$. Also, $s(x_i, z_i) = s(z_i, y_i) = \alpha$ (if $z_i = x_i = y_i$) or $s(x_i, z_i) = s(z_i, y_i) = \beta$ (if $z_i\not= x_i = y_i$). In particular, $s(x_i, z_i) = s(z_i, y_i)\le\alpha$. Since $\alpha\ge \beta$, (\ref{eqn1}) follows.
\item[iv.] In this situation, we have $h(x_i, y_i) = \beta$, $s(x_i, z_i) = \gamma$, $s(y_i, z_i) = \gamma$ and $s(z_i, z_i) = 0$. Since $\beta\ge 2\gamma$, (\ref{eqn1}) follows.
\item[viii.] In this situation, we have $h(x_i, y_i) = 0$, $s(x_i, z_i) = \gamma$, $s(y_i, z_i) = \gamma$ and $s(z_i, z_i) = \alpha$. Since $\alpha\ge 2\gamma$, (\ref{eqn1}) follows.
\end{itemize}
Since (\ref{eqn1}) holds for the triplet $(x_i, y_i, z_i)$ for all $1\le i\le n$, it follows that
\begin{eqnarray*}
s(x, y) \;\ge\; h(x, y) = \sum_{i = 1}^n h(x_i, y_i) & \ge & \sum_{i = 1}^n \Big(s(x_i, z_i) + s(z_i, y_i) - s(z_i, z_i)\Big) \\
& = & \sum_{i = 1}^n s(x_i, z_i) + \sum_{i = 1}^n s(z_i, y_i) - \sum_{i = 1}^n s(z_i, z_i) \\
& = & s(x, z) + s(z, y) - s(z, z).
\end{eqnarray*}
Hence, multiplying both sides by $-1$, we obtain that $p(x, y)\le p(x, z) + p(z, y) - p(z, z)$ for all words $x$, $y$ and $z$. Thus, $p$ is indeed a strong partial metric.
\end{proof}

\section{Topology}
Let $(X, p)$ be a partial metric space. Following \cite{SJON}, we define an {\bf open ball} as:
$$B_\epsilon(x) := \{y\mid p(x, y) - p(x, x) < \epsilon\}$$
for $x\in X$ and $\epsilon\in\bR^{> 0}$. It is easy to see that these balls are nonempty for $\epsilon > 0$, and empty for $\epsilon \le 0$. It is also easy to see that these balls form a basis for a $T_0$ topology on $X$, called the {\em pmetric topology}. We denote this topology by $\tau[p]$. Finally, since the set of positive rational numbers $\bQ^{> 0}$ is dense in $\bR^{> 0}$, it follows that every point $x\in X$ has a countable local base given by $\{B_q(x)\mid q\in\bQ^{> 0}\}$. Hence, $(X, p)$ is first countable as well.

It is noteworthy that in Definition~\ref{cauchylimit}, the definition of a limit agrees with the topological definition of a limit with respect to the topology $\tau[p]$, whereas the definition of completeness (and of special limits of Cauchy sequences) is the usual definition of completeness in the corresponding symmetrization metric topology $\tau[p^*]$ (cf.~\cite{MBRKSMHP, SJON}), i.e., $(X, p)$ is complete if and only if it is complete with respect to the metric topology $\tau[p^*]$.

We now prove the following improvement for a strong partial metric space.
\begin{thm}
For a strong partial metric space $(X, p)$, the topology $\tau[p]$ is $T_1$.
\end{thm}
\begin{proof}
Let $x, y\in X$, with $x\not= y$. By condition ({\bf sssd}), we have $p(x, x) < p(x, y)$. Consider the ball
\begin{eqnarray*}
B_{p(x, y) - p(x, x)}(x) & := & \{z\in X\mid p(x, z) - p(x, x) < p(x, y) - p(x, x)\} \\
& = & \{z\in X\mid p(x, z) < p(x, y)\}.
\end{eqnarray*}
Since $p(x, x) < p(x, y)$, we have $x\in B_{p(x, y) - p(x, x)}(x)$. But clearly $y\not\in B_{p(x, y) - p(x, x)}(x)$. Now, consider the ball $B_{p(y, x) - p(y, y)}(y)$. Similarly as above, one can show that $y\in B_{p(y, x) - p(y, y)}(y)$, but $x\not\in B_{p(y, x) - p(y, y)}(y)$. Hence, the topology $\tau[p]$ is $T_1$.
\end{proof}

We now give examples of a partial metric space that is not $T_1$ and a strong partial metric space that is not $T_2$.
\begin{example}\label{counterexample1}
Let $X := \bR\cupdot\{a\}$. Define a function $p:X\times X\to\bR$ as follows:
\begin{eqnarray*}
p(a, a) & = & 0\\
p(a, x) = p(x, a) & = & |x| \;\;\;\;\;\;\;\;\;\;\;\;\;\;\;\;\;\,\mbox{for all } x\in\bR\\
p(x, y) & = & |x - y| - 1 \;\;\;\;\;\mbox{for all } x, y\in\bR.
\end{eqnarray*}
We leave it to the reader to verify that $(X, p)$ is a partial metric space. Note, however, that $(X, p)$ is not a strong partial metric space, since $p(a, a) = 0 = p(a, 0)$, but $a\not= 0$. Also, for the same reason, $0\in B_\epsilon(a)$ for any $\epsilon\in\bR^{>0}$. Consequently, $(X, p)$ is not $T_1$.
\end{example}
\begin{example}\label{counterexample2}
Let $X := \bR^{> 0}$. Define a function $s:X\times X\to\bR$ as follows:
\begin{eqnarray*}
s(x, x) & = & x \;\;\;\;\;\;\;\;\;\;\;\mbox{for all } x\in X\\
s(x, y) & = & x + y \;\;\;\;\;\mbox{for all } x, y\in X \mbox{ with } x\not=y.
\end{eqnarray*}
Again we leave it to the reader to verify that $(X, s)$ is a strong partial metric space, but not a metric space. We now prove that $(X, s)$ is not $T_2$ via the following two claims.
\begin{claim}
For all $x\in X$ and $\epsilon\in\bR^{> 0}$, there exists $z\in X$ such that $z\not=x$ and $z\in B_\epsilon(x)$.
\end{claim}
\begin{proof}
Choose $z := \delta$ such that $0 < \delta < \epsilon$ and $\delta\not=x$. Then we have
$$s(x, z) - s(x, x) = x + z - x = z = \delta < \epsilon,$$
and hence $z\in B_\epsilon(x)$.
\end{proof}
\begin{claim}
Let $x, y, z\in X$ with $x\not=z\not=y$. Then for all $\epsilon \ge \delta > 0$, we have
$$z\in B_\delta(y) \implies z\in B_\epsilon(x).$$
\end{claim}
\begin{proof}
\begin{eqnarray*}
z\in B_\delta(y) & \implies & s(y, z) - s(y, y) < \delta \;\implies\; y + z - y < \delta \;\implies\; z < \delta \implies z < \epsilon\\
& \implies & x + z - x < \epsilon \;\implies\; s(x, z) - s(x, x) < \epsilon \;\implies\; z\in B_\epsilon(x).
\end{eqnarray*}
\end{proof}
As an immediate consequence of these two claims, it follows that for any $x, y\in X$ with $x\not=y$ and $\epsilon, \delta\in\bR^{> 0}$, we have that $B_\epsilon(x)\cap B_\delta(y)\not=\emptyset$. Hence, $(X, s)$ is not $T_2$.
\end{example}

\begin{rmk}\label{counterexample3}
As promised in the Introduction, we have given examples that show that
$$\{\mbox{Metric Spaces}\} \subsetneq \{\mbox{Strong Partial Metric Spaces}\} \subsetneq \{\mbox{Partial Metric Spaces}\}.$$
Examples~\ref{counterexample2} and \ref{counterexample1} instantiate the first and the second proper inclusion, respectively.
\end{rmk}

We end this section by giving an example of a Cauchy sequence $\langle x_n\rangle_{n\in\bN}$ in a partial metric space $(X, p)$ such that it has more than one limit in $X$, one of them being special.
\begin{example}\label{couterexample4}
Let $X := \bR\cupdot\{a\}$ and $p:X\times X\to\bR$ be as defined in Example~\ref{counterexample1}.

Set $x_n := \frac{1}{2^n}$ for $n\in\bN$. Observe that for $m > n$, we have
$$-1 \le p(x_m, x_n) = \Big|\frac{1}{2^m} - \frac{1}{2^n}\Big| - 1 = \frac{1}{2^n}\Big|\frac{1}{2^{m-n}} - 1\Big| - 1 < \frac{1}{2^n} - 1.$$
And hence, $\lim_{m, n\to\infty} p(x_m, x_n) = -1$. Thus, $\langle x_n\rangle_{n\in\bN}$ is a Cauchy sequence in $(X, p)$.

Now observe that $p(0, x_n) = \frac{1}{2^n} - 1$ for all $n\in\bN$, and hence
$$\lim_{n\to\infty} p(0, x_n) = \lim_{n\to\infty} \frac{1}{2^n} - 1 = -1 = p(0, 0) = \lim_{m, n\to\infty} p(x_m, x_n).$$
Thus, $0$ is a special limit of the sequence $\langle x_n\rangle_{n\in\bN}$.

Also observe that $p(a, x_n) = \frac{1}{2^n}$ for all $n\in\bN$, and hence
$$\lim_{n\to\infty} p(a, x_n) = \lim_{n\to\infty} \frac{1}{2^n} = 0 = p(a, a).$$
Thus, $a$ is a limit of the sequence $\langle x_n\rangle_{n\in\bN}$. Since $a\not=0$, we thus obtain non-unique limits.
\end{example}

\section{Orbital Continuity}
Recall the definition of orbital continuity, cf. Definition~\ref{orbcont}. In this section, we give an equivalent criterion for checking if the property holds in a certain special situation. We start by proving the following lemma.

\begin{lem}		\label{cauchycont}
For each partial metric space $(X, p)$, and each Cauchy sequence $\langle x_n\rangle_{n\in\bN}$ in $(X, p)$ with a special limit $a\in X$, the following holds: for every $y\in X$,
$$\lim_{n\to\infty} p(x_n, y) = p(a, y).$$
\end{lem}
\begin{proof}
Since the sequence $\langle x_n\rangle_{n\in\bN}$ is Cauchy, and $a$ is a special limit of $\langle x_n\rangle_{n\in\bN}$, we have by Definition~\ref{cauchylimit}, 
$$\lim_{m, n\to\infty}p(x_m, x_n) = \lim_{n\to\infty}p(a, x_n) = p(a, a).$$
Fix $\epsilon > 0$ and $y\in X$. For any $n\in\bN$, we have
$$p(x_n, y) \le p(x_n, a) + p(a, y) - p(a, a).$$
Choose $N_1$ large enough such that $p(x_n, a) - p(a, a) < \epsilon$ for all $n\ge N_1$. Then, for all $n\ge N_1$, we have
$$p(x_n, y) < p(a, y) + \epsilon.$$ 
Conversely, for any $n\in\bN$, we also have
$$p(a, y)\le p(a, x_n) + p(x_n, y) - p(x_n, x_n).$$
Set $\epsilon' := \frac{\epsilon}{2}$, and choose $N_2$ large enough such that $p(a, x_n) < p(a, a) + \epsilon'$ and $p(x_n, x_n) > p(a, a) - \epsilon'$ for all $n\ge N_2$. Then, for all $n\ge N_2$, we have
$$p(a, y) < (p(a, a) + \epsilon') + p(x_n, y) - (p(a, a) - \epsilon') = p(x_n, y) + 2\epsilon' = p(x_n, y) + \epsilon.$$
Setting $N := \max\{N_1, N_2\}$, we obtain for all $n\ge N$
$$p(a, y) - \epsilon < p(x_n, y) < p(a, y) + \epsilon.$$
Since $0 < \epsilon$ is arbitrary, it follows that $\lim_{n\to\infty} p(x_n, y) = p(a, y)$.
\end{proof}

As an immediate corollary, we get our promised criterion.
\begin{lem}		\label{orbcontrcont}
For each partial metric space $(X, p)$, element $x_0\in X$, and map $f:X\to X$ such that $f$ is Cauchy at $x_0$ with special limit $a\in X$, the following holds:
$$\lim_{n\to\infty} p(f^{n}x_0, fa) = p(a, fa).$$
Moreover, under the same hypothesis,
$$f \mbox{ is orbitally continuous at $x_0$ for $a$} \iff p(fa, fa) = p(a, fa).$$
\end{lem}
\begin{proof}
Setting $x_n := f^{n}x_0$ ($n\in\bN$) and $y := fa$ in Lemma~\ref{cauchycont}, we get our first statement.

Now under the same hypothesis,
\begin{eqnarray*}
& & f \mbox{ is orbitally continuous at $x_0$ for $a$} \\
& \iff & \lim_{n\to\infty} p(f^{n}x_0, fa) = p(fa, fa) \\
& \iff & p(a, fa) = p(fa, fa).
\end{eqnarray*}
The last line of the above if-and-only-if sequence follows from the first statement.
\end{proof}

We also get the uniqueness of special limits as a corollary.
\begin{lem}		\label{spllimitunique}
For each partial metric space $(X, p)$ and each Cauchy sequence $\langle x_n\rangle_{n\in\bN}$ in $(X, p)$, there is at most one special limit of $\langle x_n\rangle_{n\in\bN}$ in $X$.
\end{lem}
\begin{proof}
Let $a$ and $b$ be two special limits of $\langle x_n\rangle_{n\in\bN}$ in $X$. Then, by Definition~\ref{cauchylimit}, we have
\begin{eqnarray*}
& & \lim_{m, n\to\infty} p(x_m, x_n) = \lim_{n\to\infty} p(a, x_n) = p(a, a) \\
& & \lim_{m, n\to\infty} p(x_m, x_n) = \lim_{n\to\infty} p(b, x_n) = p(b, b).
\end{eqnarray*}

From this we obtain that $p(a, a) = p(b, b)$. By Lemma~\ref{cauchycont}, we also have
\begin{eqnarray*}
\lim_{n\to\infty} p(x_n, a) & = &  p(b, a)
\end{eqnarray*}

Combining all of these, we get that $p(a, a) = p(a, b) = p(b, b)$. By condition {\bf(sep)}, it then follows that $a = b$.
\end{proof}

We end this section with the following obvious result. Since a continuous function $f:X\to X$ on a metric space $(X, d)$ is sequentially continuous, we obtain the following.
\begin{lem}		\label{contimpliesorbcontmetric}
A continuous map $f: X\to X$ on a metric space $(X, d)$ is $\mbox{orbitally continuous.}$
\end{lem}

\section{Cauchy Mapping Theorems}
In this section, we state and prove Cauchy Mapping theorems (Theorems~\ref{cauchyorbcontnexpfpt} and \ref{strongcauchyorbcontnexpfpt}) for partial and strong partial metric spaces. These theorems serve as a basis for the proofs of various fixed point theorems in the context of partial metric spaces. Many of the known fixed point theorems for partial metric spaces have a similar pattern and the following theorems extract the essence of that pattern. \newline\newline
\underline{\bf Proof of Theorem~\ref{cauchyorbcontnexpfpt}}
\begin{proof}
We deal with the three cases separately.
\newline
{\bf Case I:} $f$ is non-expansive and orbitally continuous at $x_0$ for $a$.
\newline
Since $f$ is orbitally continuous at $x_0$ for $a$, it follows by Lemma~\ref{orbcontrcont} that 
$$p(a, fa) = p(fa, fa).$$
Since $f$ is non-expansive, we further obtain
$$p(a, fa) = p(fa, fa) \le p(a, a).$$
Combining this with condition {\bf(ssd)}, we have 
$$p(a, a) = p(a, fa) = p(fa, fa).$$ 
Hence, by condition {\bf(sep)}, it follows that $fa = a$, i.e., $a$ is a fixed point of $f$.
\newline\newline
{\bf Case II:} $f$ is orbitally continuous at $x_0$ for $a$ and $(X, p)$ is bounded below by $p(fa, fa)$.
\newline
Since $f$ is orbitally continuous at $x_0$ for $a$, it follows by Lemma~\ref{orbcontrcont} that 
$$p(a, fa) = p(fa, fa).$$
By condition {\bf(ssd)}, we have $p(a, a)\le p(a, fa)$. Since $(X, p)$ is bounded below by $p(fa, fa)$, it then follows that 
$$p(fa, fa) \le p(a, a) \le p(a, fa) = p(fa, fa).$$
Consequently, by condition {\bf(sep)}, we have that $fa = a$, i.e., $a$ is a fixed point of $f$.
\newline\newline
{\bf Case III:} $f$ is non-expansive and $(X, p)$ is bounded below by $p(a, a)$.
\newline
Since $f$ is non-expansive, we have for every $n\in\bN$,
$$p(f^{n+1}x_0, fa)\le p(f^nx_0, a).$$
By taking the limit as $n\to\infty$ and by applying Lemma~\ref{cauchycont}, we obtain
$$p(a, fa)\le p(a, a).$$
Since $(X, p)$ is bounded below by $p(a, a)$, it follows by condition {\bf(ssd)} that 
$$p(a, a) \le p(fa, fa) \le p(a, fa) \le p(a, a).$$ 
Hence, by condition {\bf(sep)}, we have $fa = a$, i.e., $a$ is a fixed point of $f$.
\end{proof}

\vspace{1em}
\hspace{-1.0em}\underline{\bf Proof of Theorem~\ref{strongcauchyorbcontnexpfpt}}
\begin{proof}
We deal with the two cases separately.
\newline
{\bf Case I:} $f$ is non-expansive.
\newline
Since $f$ is non-expansive, we have for every $n\in\bN$,
$$p(f^{n+1}x_0, fa)\le p(f^nx_0, a).$$
By taking the limit as $n\to\infty$ and by applying Lemma~\ref{cauchycont}, we obtain
$$p(a, fa)\le p(a, a).$$
Since $(X, p)$ is strong, it then follows from condition ({\bf sssd}) that $fa = a$.
\newline\newline
{\bf Case II:} $f$ is orbitally continuous at $x_0$ for $a$.
\newline
Since $f$ is orbitally continuous at $x_0$ for $a$, it follows by Lemma~\ref{orbcontrcont} that 
$$p(a, fa) = p(fa, fa).$$
Since $(X, p)$ is strong, it then follows from condition ({\bf sssd}) that $fa = a$.
\end{proof}

We now give an example of a function $f$ on the complete metric space $(\bR, |\cdot|)$ that is non-expansive on an interval and Cauchy on one orbit in that interval, but is not contractive on the interval or even on that particular orbit. This function has a fixed point in $\bR$ by our theorem, but it cannot be obtained by the Banach Contraction Mapping Theorem.
\begin{example}
Consider the following function $f:\bR\to\bR$ given by
$$fx = \frac{e^x\sin x}{e^{\pi/2}} + \frac{\pi}{2} - 1.$$ 
It is easy to check that $f\frac{\pi}{2} = \frac{\pi}{2}$, i.e. $x=\frac{\pi}{2}$ is a fixed point of $f$. We first note that $f$ is not contractive on any interval containing the point $\frac{\pi}{2}$, and so we cannot apply Banach's Fixed Point Theorem to obtain this fixed point. Observe that $f$ is infinitely differentiable and the first three derivatives of $f$ are given by
\begin{eqnarray*}
f'x = \frac{e^x(\sin x + \cos x)}{e^{\pi/2}} \;\;\;\;\;\;\;\;\;\; f''x = \frac{2e^x\cos x}{e^{\pi/2}} \;\;\;\;\;\;\;\;\;\; f'''x = \frac{2e^x(\cos x - \sin x)}{e^{\pi/2}}.
\end{eqnarray*}
Solving the second equation for a root, one then obtains the following:
\begin{eqnarray*}
f''\frac{\pi}{2} = 0 & \;\;\;\;\;\; & f'''\frac{\pi}{2} = -2 < 0.
\end{eqnarray*}
It follows by elementary calculus that $f'$ has a local maximum at the point $x = \frac{\pi}{2}$ with a value of $f'\frac{\pi}{2} = 1$. Since $f'$ is continuous, $f'\frac{\pi}{2} = 1$, and for sufficiently small interval $[a, b]$ around $\frac{\pi}{2}$ and for any $x < y\in [a, b]$ there exists $\xi\in(x, y)$ such that
$$|fx - fy| = |f'\xi||x - y|,$$
it follows that $f$ is not a contractive function on any interval containing $\frac{\pi}{2}$.

However, $(R, | \cdot |)$ is a metric space and hence a strong partial metric space. The function $f$ is continuous, and thus by Lemma~\ref{contimpliesorbcontmetric}, is orbitally continuous. Also observe that
\begin{eqnarray*}
f'\frac{3\pi}{4} = 0 \;\;\;\;\;\;\;\;\; f'x > 0 \mbox{ for } x\in \Big[0, \frac{3\pi}{4}\Big) \;\;\;\;\;\;\;\;\; 0 < f0 < f\frac{3\pi}{4} < \frac{3\pi}{4}.
\end{eqnarray*}
In particular, $f$ is increasing on the interval $[0, \frac{3\pi}{4}]$. Set $x_0 := 0$. Then we have
$$x_0 < fx_0 < f\frac{3\pi}{4} < \frac{3\pi}{4}.$$
Since $f$ is increasing on $[0, \frac{3\pi}{4}]$, it then follows by simple induction that 
$$x_0 < fx_0 < f^2x_0 < f^3x_0 < \ldots < \frac{3\pi}{4}.$$
Thus, we get a bounded increasing sequence $\langle f^nx_0\rangle_{n\in\bN}$ in $\bR$. By the Monotone Convergence Theorem, this sequence is Cauchy and converges to some $x_1\in\bR$ such that $x_1\in [0, \frac{3\pi}{4}]$. By Theorem~\ref{strongcauchyorbcontnexpfpt}, it follows that $x_1$ is a fixed point of $f$. A numerical simulation then shows that indeed $x_1 = \frac{\pi}{2}$.
\end{example}

It should be noted that in the above example, $f$ is a continuous function on the interval $[0, \frac{3\pi}{4}]$ with $0 < f0$ and $f\frac{3\pi}{4} < \frac{3\pi}{4}$. Thus, the existence of a fixed point of $f$ in the interval $[0, \frac{3\pi}{4}]$ follows immediately from the Intermediate Value Theorem applied to the function $gx = fx - x$. However, what Theorem~\ref{strongcauchyorbcontnexpfpt} (specialized to the case of metric spaces) does additionally is that it gives an iterative method for computing that fixed point starting from a nearby point. Also it is easy to see that $x = \frac{\pi}{2}$ is the unique maxima of $f'$ in the interval $[0, \frac{3\pi}{4}]$, and thus $f$ is non-expansive on the interval $[0, \frac{3\pi}{4}]$.\newline

We like to end this section with the following remark.
\begin{rmk}
The assumptions of lower bounds of $(X, p)$ in conditions (2) and (3) of Theorem~\ref{cauchyorbcontnexpfpt} look artificial and probably often untenable in the grand scheme of things. But the reason we have listed them is to emphasize their analogy with what is going on under the hood in partial metric spaces bounded below by zero. In these situations, people often consider a contractive map (or at the very least a non-expansive and orbitally contractive map) which has the property that $\lim_{m, n\to\infty} p(f^mx_0, f^nx_0) = 0$ (for some $x_0$), where zero is incidentally the lower bound of the space, and one ends up getting a fixed point for similar reasons as explained in the proof of our theorem. The point we are trying to make here is that the lower bound of zero is a subtle third condition in such theorems, and that, in its absence, orbital continuity is probably the right alternative to fall back on. Even nicer is the fact that this whole point is moot if we are in a strong partial metric space because then the existence or the non-existence of a lower bound of the space has no effect on the existence of a fixed point as Theorem~\ref{strongcauchyorbcontnexpfpt} shows.
\end{rmk}

\section{Orbitally $r$-contractive maps}
Let $(X, p)$ be a partial metric space. In the previous section, we showed the existence of a fixed point for a function $f:X\to X$ under the assumption that there is an element $x_0\in X$ such that $f$ is Cauchy at $x_0$. In this section, we give an example of a particular class of functions, which we call ``orbitally $r$-contractive'', that in fact satisfies this condition. Lemma~\ref{orbcontrCauchy} establishes this claim. These functions are our analogues of contractive (rather, orbitally contractive) functions suitable to our context.
\begin{defn}		\label{selfcontmap}
Take a partial metric space $(X, p)$, an element $x_0\in X$, a number $r\in\bR$, and a map $f: X\to X$. We say $f$ is {\em orbitally $r$-contractive at $x_0$} if there exists a real number $c$ with $0\le c < 1$ such that the following two conditions hold for all $n\in\bN$:
\begin{itemize}
\item $r\le p(f^nx_0, f^nx_0)$
\item $p(f^{n+2}x_0, f^{n+1}x_0) \le r + c^{n+1}\,|p(fx_0, x_0)|$.
\end{itemize}
And we say $f$ is {\em orbitally $r$-contractive} if $f$ is orbitally $r$-contractive at every $x\in X$.
\end{defn}

Observe that if $(X, p)$ is bounded below by zero and $f$ is contractive, then $f$ is orbitally $0$-contractive, but not conversely. Thus, orbitally $r$-contractive maps form a strictly wider class than contractive maps even for partial metric spaces bounded below by zero. 
\begin{lem}		\label{orbcontrCauchy}
For each partial metric space $(X, p)$, element $x_0\in X$, real number $r\in\bR$, and map $f:X\to X$ orbitally $r$-contractive at $x_0$, the orbit $\langle f^nx_0\rangle_{n\in\bN}$ of $x_0$ under $f$ is a Cauchy sequence in $(X, p)$ with $\lim_{m, n\to\infty}p(f^mx_0, f^nx_0) = r$.
\end{lem}
\begin{proof}
Since $f$ is orbitally $r$-contractive at $x_0$, there is $0\le c < 1$ such that for all $n\in\bN$
\begin{eqnarray*}
& & r\le p(f^nx_0, f^nx_0) \\
& & p(f^{n+2}x_0, f^{n+1}x_0) \le r + c^{n+1}\,|p(fx_0, x_0)|.
\end{eqnarray*}
Let $m > n\ge 0$ be arbitrary. Write $m = n+k+1$ for some $k \ge 0$. Then we have
\begin{eqnarray*}
p(f^{n+k+1}x_0, f^nx_0) & \le & p(f^{n+k+1}x_0, f^{n+k}x_0) + p(f^{n+k}x_0, f^nx_0) - p(f^{n+k}x_0, f^{n+k}x_0) \\
& \le & r + c^{n+k}\,|p(fx_0, x_0)| + p(f^{n+k}x_0, f^nx_0) - r \\
& \le & c^{n+k}\,|p(fx_0, x_0)| + p(f^{n+k}x_0, f^nx_0) \\
& \le & \cdots \\
& \le & (c^{n+k} + \cdots + c^{n+1})\,|p(fx_0, x_0)| + p(f^{n+1}x_0, f^nx_0) \\
& \le & (c^{n+k} + \cdots + c^{n+1})\,|p(fx_0, x_0)| + r + c^n|p(fx_0, x_0)| \\
& \le & (c^{n+k} + \cdots + c^{n+1} + c^n)\,|p(fx_0, x_0)| + r \\
& \le & r + c^n\dfrac{1 - c^{k+1}}{1 - c}\,|p(fx_0, x_0)| \\
& \le & r + c^n\dfrac{1}{1 - c}\,|p(fx_0, x_0)|.
\end{eqnarray*}
Taking the limit as $n\to\infty$, the right hand side of the above inequality goes to $r$, since $0\le c < 1$. Since also $r\le p(f^nx_0, f^nx_0) \le p(f^{n+k+1}x_0, f^nx_0)$ for all $k, n\in\bN$, we have
$$\lim_{m, n\to\infty}p(f^mx_0, f^nx_0) = r,$$
and hence $\langle f^nx_0\rangle_{n\in\bN}$ is Cauchy.
\end{proof}

It thus follows that an orbitally $r$-contractive map is Cauchy for every $r\in\bR$. In particular, if $(X, p)$ is bounded below by zero and $f:X\to X$ is contractive, then $f$ is Cauchy (since $f$ is orbitally 0-contractive). Because of this property, the orbitally $r$-contractive functions provide more examples of fixed point theorems. But for that we need the existence of special limits. The following weakening of completeness suffices for our fixed point theorems to work. 
\begin{defn}		\label{orbcomp}
Given a partial metric space $(X, p)$ and a map $f:X\to X$, the space $(X, p)$ is called {\em orbitally complete for $f$} if every Cauchy sequence in $(X, p)$ of the form $\langle f^nx_0\rangle_{n\in\bN}$, for $x_0\in X$, has a special limit $a\in X$.
\end{defn}

Combining this with the results of the previous section, we obtain the following.
\begin{thm}		\label{cauchyorbcontractivefpt}
Let $(X, p)$ be a partial metric space, $r\in\bR$, $x_0\in X$ and $f:X\to X$ be a map such that $f$ is orbitally $r$-contractive at $x_0$ and $(X, p)$ is orbitally complete for $f$. Further assume that one of the following holds:
\begin{enumerate}
\item $f$ is non-expansive and orbitally continuous at $x_0$;
\item $f$ is non-expansive and $(X, p)$ is bounded below by $r$.
\end{enumerate}
Then there exists $a\in X$ such that $fa = a$ and $p(a, a) = r$. Furthermore, if $(X, p)$ is bounded below by zero and $f$ is contractive, then the fixed point is unique.
\end{thm}
\begin{proof}
By Lemma~\ref{orbcontrCauchy}, the orbit $\langle f^nx_0\rangle_{n\in\bN}$ of $x_0$ under $f$ is a Cauchy sequence with $\lim_{m, n\to\infty} p(f^mx_0, f^nx_0) = r$. Since $(X, p)$ is orbitally complete for $f$, there is an element $a\in X$ such that $a$ is a special limit of $\langle f^nx_0\rangle_{n\in\bN}$. By Definition~\ref{cauchylimit}, we have
$$p(a, a) = \lim_{m, n\to\infty} p(f^mx_0, f^nx_0) = r.$$ 
Finally, by Theorem~\ref{cauchyorbcontnexpfpt}, we have $fa = a$, i.e., $a$ is a fixed point of $f$.

Now suppose $(X, p)$ is bounded below by zero and $f$ is contractive. Suppose $a, b\in X$ are two fixed points of $f$. Then
$$p(a, b) = p(fa, fb)\le c\,p(a, b),$$
where $c$ is a real number such that $0\le c < 1$. If $c = 0$, then it implies $p(a, b) = 0$. If $0 < c < 1$, then $p(a, b)\not=0$ implies
$$p(a, b) < p(a, b),$$
which is absurd. Hence, in all cases, we have $p(a, b) = 0$, which implies $a = b$.
\end{proof}

An analogous proof using Theorem~\ref{strongcauchyorbcontnexpfpt} instead of Theorem~\ref{cauchyorbcontnexpfpt} then gives the following.
\begin{thm}		\label{strongcauchyorbcontractivefpt}
Let $(X, p)$ be a strong partial metric space, $r\in\bR$, $x_0\in X$ and $f:X\to X$ be a map such that $f$ is orbitally $r$-contractive at $x_0$ and $(X, p)$ is orbitally complete for $f$. Further assume that one of the following holds:
\begin{enumerate}
\item $f$ is non-expansive;
\item $f$ is orbitally continuous at $x_0$.
\end{enumerate}
Then there exists $a\in X$ such that $fa = a$ and $p(a, a) = r$.
\end{thm}

We end this section with the following comment. Note the three main differences between Theorem~\ref{cauchyorbcontractivefpt} and Theorem~\ref{contmapthm}: first, we have achieved $p(a, a) = r$ instead of $p(a, a) = 0$; second, we have lost the uniqueness of a fixed point; and third, and most importantly, we have got rid of the ``bounded below by zero'' condition. In other words, our theorem works for partial metric spaces with negative distances.

\section{Orbitally $\phi_r$-contractive maps}
In this section, we define another class of functions, which we call ``orbitally $\phi_r$-contractive'', that also satisfies the property of being Cauchy and for which we get similar fixed point theorems. Lemma~\ref{phircontmaplem} establishes this claim. Analogous functions have been studied by several authors. Boyd and Wong \cite{DWBSWW} introduced the notion of ``$\Phi$-contraction'' on a metric space $(X, d)$ as a map $f:X\to X$ for which there exists an upper semi-continuous function $\Phi : [0, \infty)\to [0, \infty)$ such that
$$d(fx, fy) \le \Phi(d(x, y))\;\;\;\;\;\mbox{ for all } x, y\in X.$$ 
Alber and Guerre-Delabriere \cite{YIASGD} generalized the notion of $\Phi$-contraction by defining the notion of ``weak $\phi$-contraction'' for an Hilbert space $(X, d)$ as a map $f:X\to X$ for which there exists a strictly increasing map $\phi : [0, \infty)\to [0, \infty)$ with $\phi(0) = 0$ such that
$$d(fx, fy) \le d(x, y) - \phi(d(x, y))\;\;\;\;\;\mbox{ for all } x, y\in X.$$
Karapinar \cite{EK2} extended this definition to partial metric spaces bounded below by zero. In this paper, we generalize this further to all partial metric spaces.

\begin{defn}	\label{phircont}
Take a partial metric space $(X, p)$, an element $x_0\in X$, a number $r\in\bR$, and a map $f:X\to X$. We say $f$ is {\em orbitally $\phi_r-$contractive at $x_0$} if there exists a continuous non-decreasing function $\phi: [r, \infty)\to [0, \infty)$ with $\phi(r) = 0$ and $\phi(t) > 0$ for all $t > r$ such that the following two conditions hold for all $m, n\in\bN$:
\begin{itemize}
\item $r \le p(f^nx_0, f^nx_0)$
\item $p(f^{m+1}x_0, f^{n+1}x_0) \le p(f^mx_0, f^nx_0) - \phi(p(f^mx_0, f^nx_0))$.
\end{itemize}
We say $f$ is {\em orbitally $\phi_r-$contractive} if it is orbitally $\phi_r-$contractive at every $x\in X$.
\end{defn}
It is clear from the above definition that, in the context of a partial metric space bounded below by zero, a weak $\phi$-contraction map in the sense of \cite{EK2} is an orbitally $\phi_0$-contractive map in our sense (but not conversely).




\begin{lem}		\label{phircontmaplem}
For each partial metric space $(X, p)$, element $x_0\in X$, real number $r\in\bR$, and map $f:X\to X$ orbitally $\phi_r$-contractive at $x_0$, the orbit $\langle f^nx_0\rangle_{n\in\bN}$ of $x_0$ under $f$ is a Cauchy sequence in $(X, p)$ with $\lim_{m,n\to\infty} p(f^mx_0, f^nx_0) = r$.
\end{lem}
\begin{proof}
Set $x_{n+1} := fx_n$ for $n\in\bN$. \newline 
Let $\phi: [r, \infty)\to [0, \infty)$ witness the fact that $f$ is orbitally $\phi_r-$contractive at $x_0$. In particular, $\phi$ is a continuous non-decreasing function with $\phi(r) = 0$ and $\phi(t) > 0$ for all $t > r$. It follows that
$$r\le p(x_{n+2}, x_{n+1}) = p(fx_{n+1}, fx_n) \le p(x_{n+1}, x_n) - \phi(p(x_{n+1}, x_n)).$$
Set $t_n := p(x_{n+1}, x_n)$. Then one obtains
\begin{eqnarray}		\label{tnseq}
r\le t_{n+1}\le t_n - \phi(t_n) \le t_n.
\end{eqnarray}
This implies that $\langle t_n\rangle_{n\in\bN}$ is a non-increasing sequence of real numbers bounded below by $r$, and hence converges to some $L\ge r$. We claim that $L = r:$ otherwise $L > r$, and hence $\phi(L) > 0$. Since $\phi$ is non-decreasing, we get $\phi(L) \le \phi(t_n)$ for all $n\in\bN$. Due to (\ref{tnseq}), we have $t_{n+1} \le t_n - \phi(t_n) \le t_n - \phi(L)$, and so
$$t_{n+2} \le t_{n+1} - \phi(t_{n+1}) \le t_n - \phi(t_n) - \phi(t_{n+1}) \le t_n - 2\phi(L).$$
Inductively we obtain $t_{n+k} \le t_n - k\phi(L)$, which is a contradiction for large enough $k\in\bN$. Thus, we have $\phi(L) = 0$, and hence $L = r$. Consequently, $\lim_{n\to\infty} p(x_{n+1}, x_n) = r$.

Now we show that $\langle x_n\rangle_{n\in\bN}$ is a Cauchy sequence in $(X, p)$ with $\lim_{m, n\to\infty} p(x_m, x_n) = r$. Suppose it is not the case. Then there exists $\epsilon_0 > r$ and two sequences of integers $\langle n(k)\rangle_{k\in\bN}$ and $\langle m(k)\rangle_{k\in\bN}$ such that $m(k) > n(k)\ge k$ and
\begin{eqnarray}		\label{greatereps}
s_k := p(x_{m(k)}, x_{n(k)}) \ge\epsilon_0
\end{eqnarray}
for all $k\in\bN$. We also assume, for each $k$, that $m(k)$ is the smallest number exceeding $n(k)$ for which (\ref{greatereps}) holds. In particular, $p(x_{m(k) - 1}, x_{n(k)}) < \epsilon_0$ for $k\in\bN$. Thus, we have
\begin{eqnarray*}
\epsilon_0 \;\le\; s_k & = & p(x_{m(k)}, x_{n(k)}) \\
& \le & p(x_{m(k)}, x_{m(k) - 1}) + p(x_{m(k) - 1}, x_{n(k)}) - p(x_{m(k) - 1}, x_{m(k) - 1}) \\
& \le & t_{m(k) - 1} + \epsilon_0 - r \\
& \le & t_k + \epsilon_0 - r.
\end{eqnarray*}
Since $\lim_{k\to\infty} t_k = r$, we have $\lim_{k\to\infty} (t_k + \epsilon_0 - r) = r + \epsilon_0 - r = \epsilon_0$. Consequently,
$$\lim_{k\to\infty} s_k = \epsilon_0.$$ 
On the other hand,
\begin{eqnarray*}
s_k & = & p(x_{m(k)}, x_{n(k)}) \\
& \le & p(x_{m(k)}, x_{m(k) + 1}) + p(x_{m(k) + 1}, x_{n(k)}) - p(x_{m(k) + 1}, x_{m(k) + 1}) \\
& \le & t_{m(k)} - r + p(x_{m(k) + 1}, x_{n(k)}) \\
& \le & t_{m(k)} - r + p(x_{m(k) + 1}, x_{n(k) + 1}) + p(x_{n(k) + 1}, x_{n(k)}) - p(x_{n(k) + 1}, x_{n(k) + 1}) \\
& \le & t_{m(k)} - r + t_{n(k)} - r + p(x_{m(k) + 1}, x_{n(k) + 1}) \\
& \le & 2t_k - 2r + p(x_{m(k) + 1}, x_{n(k) + 1}) \\
& \le & 2t_k - 2r + p(x_{m(k)}, x_{n(k)}) - \phi(p(x_{m(k)}, x_{n(k)})) \\
& \le & 2t_k - 2r + s_k - \phi(s_k) \\
\implies \phi(s_k) & \le & 2t_k - 2r
\end{eqnarray*}
Again, since $\lim_{k\to\infty} t_k = r$, we have $\lim_{k\to\infty} (2t_k - 2r) = 2r - 2r = 0$. Since $\phi(s_k)\ge 0$ for all $k\in\bN$ (by the definition of $\phi$), we get that $\lim_{k\to\infty} \phi(s_k) = 0$. Since $\phi$ is continuous, it then follows that
$$0 = \lim_{k\to\infty}\phi(s_k) = \phi\Big(\lim_{k\to\infty} s_k\Big) = \phi(\epsilon_0),$$
which contradicts the fact that $\epsilon_0 > r$. Hence, $\langle x_n\rangle_{n\in\bN}$ is Cauchy.
\end{proof}

It thus follows that an orbitally $\phi_r$-contractive map is Cauchy for every $r\in\bR$. A similar claim for weak $\phi$-contractive functions in the context of partial metric spaces bounded below by zero has been proved in \cite[Theorem 2.2]{EK2}. However, the proof there is incomplete. The author proved the claim under the assumption that $s_n := \sup\{p(x_i, x_j)\mid i, j\ge n\}$ exists for all $n$, for which he did not provide any justification. Our proof above is different and does not assume/require the existence of such a supremum.

Because of the Cauchy property, the orbitally $\phi_r$-contractive functions provide further examples of fixed point theorems. As a corollary, we obtain the following generalization of \cite[Theorem 2.2]{EK2}.
\begin{thm}\label{phircontmapthm}
Let $(X, p)$ be a partial metric space, $r\in\bR$, $x_0\in X$, $\phi:[r, \infty)\to [0, \infty)$ be a continuous and non-decreasing function with $\phi(r ) = 0$ and $\phi(t) > 0$ for all $t > r$, and $f:X\to X$ be a map such that $f$ is orbitally $\phi_r$-contractive at $x_0$ and $(X, p)$ is orbitally complete for $f$. Further assume that one of the following holds:
\begin{enumerate}
\item $f$ is non-expansive and orbitally continuous at $x_0$; 
\item $f$ is non-expansive and $(X, p)$ is bounded below by $r$.
\end{enumerate}
Then there exists $a\in X$ such that $fa = a$ and $p(a, a) = r$.
\end{thm}
\begin{proof}
By Lemma~\ref{phircontmaplem}, the orbit $\langle f^nx_0\rangle_{n\in\bN}$ of $x_0$ under $f$ is a Cauchy sequence with $\lim_{m,n\to\infty} p(f^mx_0,f^nx_0) = r$. Since $(X, p)$ is orbitally complete for $f$, there is an element $a\in X$ such that $a$ is a special limit of $\langle f^nx_0\rangle_{n\in\bN}$. By Definition~\ref{cauchylimit}, we have
$$p(a, a) = \lim_{m,n\to\infty} p(f^mx_0, f^nx_0) = r.$$
Finally, by Theorem~\ref{cauchyorbcontnexpfpt}, we have $fa = a$, i.e., a is a fixed point of $f$.
\end{proof}

An analogous proof using Theorem~\ref{strongcauchyorbcontnexpfpt} instead of Theorem~\ref{cauchyorbcontnexpfpt} then gives the following.
\begin{thm}
Let $(X, p)$ be a strong partial metric space, $r\in\bR$, $x_0\in X$, $\phi:[r, \infty)\to [0, \infty)$ be a continuous and non-decreasing function with $\phi(r ) = 0$ and $\phi(t) > 0$ for all $t > r$, and $f:X\to X$ be a map such that $f$ is orbitally $\phi_r$-contractive at $x_0$ and $(X, p)$ is orbitally complete for $f$. Further assume that one of the following holds:
\begin{enumerate}
\item $f$ is non-expansive;
\item $f$ is orbitally continuous at $x_0$.
\end{enumerate}
Then there exists $a\in X$ such that $fa = a$ and $p(a, a) = r$.
\end{thm}

\section{Variations}
In this final section, we apply our results to give alternate proofs to two other known fixed point theorems \cite{EKIE, LBC} in the special case when the underlying space is a strong partial metric space. In this context, we see that the use of Theorem~\ref{strongcauchyorbcontractivefpt} makes the proofs of these theorems non-repetitive and more streamlined. We start with the following theorem.

\begin{thm}
Let $(X, p)$ be a strong partial metric space bounded below by zero and $f:X\to X$ be an orbitally continuous map such that $(X, p)$ is orbitally complete for $f$. If there is some $0 < c < 1$ such that
$$\min\{p(fx, fy),\; p(x, fx),\; p(y, fy)\} \le c\,p(x, y)$$
for all $x, y \in X$, then the sequence $\langle f^nx\rangle_{n\in\bN}$ converges to a fixed point of $f$ with self distance 0, for every $x\in X$.
\end{thm}
\begin{proof}
Pick any $x_0\in X$. We first show that $f$ is orbitally 0-contractive at $x_0:$ consider the orbit $\langle f^nx_0\rangle_{n\in\bN}$ of $x_0$ under $f$. Set $x_n := f^nx_0$, for $n\in\bN$. If $p(x_n, x_{n+1}) = 0$ for some $n$, then $x_n = x_{n+1}$, and consequently, $x_m = x_n$ for all $m\ge n$. Moreover, $p(x_n, x_n) = p(x_n, x_{n+1}) = 0$. Hence, we have got our result. So assume without loss of generality that $p(x_n, x_{n+1})\not= 0$ for all $n\in\bN$. Letting $x = x_n$ and $y = x_{n+1}$ in the given condition for $f$, we get
\begin{eqnarray*}
\min\{p(x_{n+1}, x_{n+2}),\; p(x_n, x_{n+1}),\; p(x_{n+1}, x_{n+2})\} & \le & c\,p(x_n, x_{n+1}).
\end{eqnarray*}
If $\min\{p(x_{n+1}, x_{n+2}),\; p(x_n, x_{n+1})\} = p(x_n, x_{n+1})$, then it implies
$$p(x_n, x_{n+1}) \le c\,p(x_n, x_{n+1}) < p(x_n, x_{n+1}),$$
which is absurd. Hence, $\min\{p(x_{n+1}, x_{n+2}),\; p(x_n, x_{n+1})\} = p(x_{n+1}, x_{n+2})$, and therefore,
$$p(x_{n+1}, x_{n+2})\le c\,p(x_n, x_{n+1}).$$
Thus, $f$ is orbitally 0-contractive at $x_0$. The rest follows by Theorem~\ref{strongcauchyorbcontractivefpt}.
\end{proof}

Now we give our second application.
\begin{thm}
Let $(X, p)$ be a strong partial metric space bounded below by zero and $f:X\to X$ be an orbitally continuous map such that $(X, p)$ is orbitally complete for $f$. If there is some $0 < c < 1$ such that
$$\dfrac{\min\{p(fx, fy)p(x, y),\; p(x, fx)p(y, fy)\}}{\min\{p(x, fx),\; p(y, fy)\}} \le c\,p(x, y)$$
for all $x, y \in X$ such that $p(x, fx)\not= 0$ and $p(y, fy)\not= 0$, then the sequence $\langle f^nx\rangle_{n\in\bN}$ converges to a fixed point of $f$ with self distance 0, for every $x\in X$.
\end{thm}
\begin{proof}
Pick any $x_0\in X$. We start by showing that $f$ is orbitally 0-contractive at $x_0:$ consider the orbit $\langle f^nx_0\rangle_{n\in\bN}$ of $x_0$ under $f$. Set $x_n := f^nx_0$, for $n\in\bN$. If $p(x_n, x_{n+1}) = 0$ for some $n$, then $x_n = x_{n+1}$, and consequently, $x_m = x_n$ for all $m\ge n$. Moreover, $p(x_n, x_n) = p(x_n, x_{n+1}) = 0$. Hence, we have got our result. So assume without loss of generality that $p(x_n, x_{n+1})\not= 0$ for all $n\in\bN$. Letting $x = x_n$ and $y = x_{n+1}$ in the given condition for $f$, we get
\begin{eqnarray*}
& \dfrac{\min\{p(x_{n+1}, x_{n+2})p(x_n, x_{n+1}),\; p(x_n, x_{n+1})p(x_{n+1}, x_{n+2})\}}{\min\{p(x_n, x_{n+1}),\; p(x_{n+1}, x_{n+2})\}} & \le c\,p(x_n, x_{n+1}) \\
\implies & \dfrac{p(x_{n+1}, x_{n+2})p(x_n, x_{n+1})}{\min\{p(x_n, x_{n+1}),\; p(x_{n+1}, x_{n+2})\}} & \le c\,p(x_n, x_{n+1}).
\end{eqnarray*}
If $\min\{p(x_n, x_{n+1}),\; p(x_{n+1}, x_{n+2})\} = p(x_{n+1}, x_{n+2})$, then it implies
$$p(x_n, x_{n+1}) \le c\,p(x_n, x_{n+1}) < p(x_n, x_{n+1}),$$
which is absurd. Hence, $\min\{p(x_n, x_{n+1}),\; p(x_{n+1}, x_{n+2})\} = p(x_n, x_{n+1})$, and therefore,
$$p(x_{n+1}, x_{n+2})\le c\,p(x_n, x_{n+1}).$$
Thus, $f$ is orbitally 0-contractive at $x_0$. The rest follows by Theorem~\ref{strongcauchyorbcontractivefpt}.
\end{proof}

\bibliographystyle{amsplain}
\bibliography{references}

\end{document}